\documentclass[10pt]{article}
\usepackage{amsfonts}
\usepackage{graphicx}
\usepackage{amsmath}
\usepackage{amsthm}
\usepackage{amssymb}
\usepackage{amscd}
\usepackage{graphicx}
\usepackage{microtype}
\usepackage{enumerate}
\usepackage[top=1.3in,bottom=1.5in,left=1.3in,right=1.3in]{geometry}
\theoremstyle{definition}
\newtheorem{theorem}{Theorem}[section]

\newtheorem{lemma}[theorem]{Lemma}

\newtheorem{remark}[theorem]{Remark}

\newcommand{\R}{\mathbb{R}}
\newcommand{\C}{\mathbb{C}}
\newcommand{\Z}{\mathbb{Z}}

\newcommand{\Q}{\mathbb{Q}}

\DeclareMathOperator{\PSL}{PSL}

\DeclareMathOperator{\SL}{SL}

\title{A counterexample to the simple loop conjecture for $\PSL(2,\R)$}
\author{Kathryn Mann}
\date{}

\begin{document}

\maketitle

\begin{abstract}
In this note, we give an explicit counterexample to the simple loop conjecture for representations of surface groups into $\PSL(2,\R)$.  Specifically, we show that for any surface with negative Euler characteristic and genus at least 1, there are uncountably many non-conjugate, non-injective homomorphisms of its fundamental group into $\PSL(2,\R)$ that kill no simple closed curve (nor any power of a simple closed curve).  This result is not new -- work of Louder and Calegari for representations of surface groups into $\SL(2, \C)$ applies to the $\PSL(2,\R)$ case, but our approach here is explicit and elementary.

\end{abstract}

\setcounter{section}{0}

\section{Introduction} 

The simple loop conjecture, proved by Gabai in \cite{Ga}, states that any non-injective homomorphism from a closed surface group to another closed surface group has an element represented by a simple closed curve in the kernel.  It has been conjectured that the result still holds if the target is replaced by the fundamental group of an orientable 3-manifold (see Kirby's problem list in \cite{Ki}).  Although special cases have been proved (e.g. \cite{Ha}, \cite{RW}), the general hyperbolic case is still open. 

Recently, Cooper and Manning showed that if instead of a 3-manifold group the target group is $\SL(2, \C)$, then the conjecture is false.  Precisely, they show: 

\begin{theorem}[Cooper-Manning \cite{CM}]  \label{cmthm} Let $\Sigma$ be a closed orientable surface of genus $g \geq 4$.  Then there is a homomorphism $\rho: \pi_1(\Sigma) \to \SL(2, \C)$ such that 
\begin{enumerate} 
\item $\rho$ is not injective
\item If $\rho(\alpha) = \pm I$, then $\alpha$ is not represented by a simple closed curve
\item If $\rho(\alpha)$ has finite order, then $\rho(\alpha) = I$
\end{enumerate}
The third condition implies in particular that no \emph{power} of a simple closed curve lies in the kernel.  
\end{theorem}

Inspired by this, we asked whether a similar result holds for $\PSL(2,\R)$, this being an intermediate case between Gabai's result for surface groups and Cooper and Manning's for $\SL(2,\C)$.  Cooper and Manning's proof uses a dimension count on the $\SL(2,\C)$ character variety and a proof that a specific subvariety is irreducible and smooth on a dense subset, much of which does not carry over to the $\PSL(2, \R)$ case.  In general, complex varieties and their real points can behave quite differently.   However, we show here with different methods that an analogous result does hold.  

While this note was in progress, we learned of work of Louder and Calegari (independently in \cite{Lo} and \cite{Ca}) that can also be applied to answer our question in the affirmative. Louder shows the simple loop conjecture is false for representations into limit groups, and Calegari gives a practical way of verifying no simple closed curves lie in the kernel of a non-injective representation using stable commutator length and the Gromov norm.  

The difference here is that our construction is entirely elementary.  We use an explicit representation from DeBlois and Kent in \cite{DK} and verify that this representation it is non injective and kills no simple closed curve by elementary means.   Our end result parallels that of Cooper and Manning but also include surfaces with boundary and all genera at least 1: 

\begin{theorem}\label{main}
Let $\Sigma$ a surface of negative Euler characteristic and of genus $g \geq 1$
, possibly with boundary.  
Then there is a homomorphism $\rho: \pi_1(\Sigma) \to \SL(2, \R)$ such that 
\begin{enumerate} 
\item $\rho$ is not injective
\item If $\rho(\alpha) = \pm I$, then $\alpha$ is not represented by a simple closed curve
\item In fact, if $\alpha$ is represented by a simple closed curve, then $\rho(\alpha^k) \neq 1$ for any $k \in \Z$.  
\end{enumerate}
Moreover, there are uncountably many non-conjugate representations satisfying 1. through 3.  
\end{theorem}

\section{Proof of theorem \ref{main}} 

We first present a construction of a (non-injective) representation from DeBlois and Kent in \cite{DK}, and then show that no power of a simple closed curve lies in the kernel of this representation.  The full construction appears in \cite{DK}, we describe it here for convenience. 

Let $\Sigma$ be a surface of genus $g \geq 1$ and negative Euler characteristic, possibly with boundary. Assume for the moment that $\Sigma$ is not the once-puntured torus -- Theorem \ref{main} for this case will follow easily later on.  

Let $c \subset \Sigma$ be a simple closed curve separating $\Sigma$ into a genus 1 subsurface with single boundary component $c$, and a genus ($g-1$) subsurface with one or more boundary components.  Let $\Sigma_A$ denote the genus $(g-1)$ subsurface and $\Sigma_B$ the genus 1 subsurface.  See Figure \ref{setup} below.  Finally, we let $A= \pi_1(\Sigma_A)$ and $B = \pi_1(\Sigma_B)$, so that $\pi_1(\Sigma) = A \ast_C B$, where $C$ is the $\Z$-subgroup generated by the element $[c]$ represented by the curve $c$.  We assume that the basepoint for $\pi_1(\Sigma)$ lies on $c$.   

  \begin{figure*}[h]
  \centerline{
    \mbox{\includegraphics[width=2.5in]{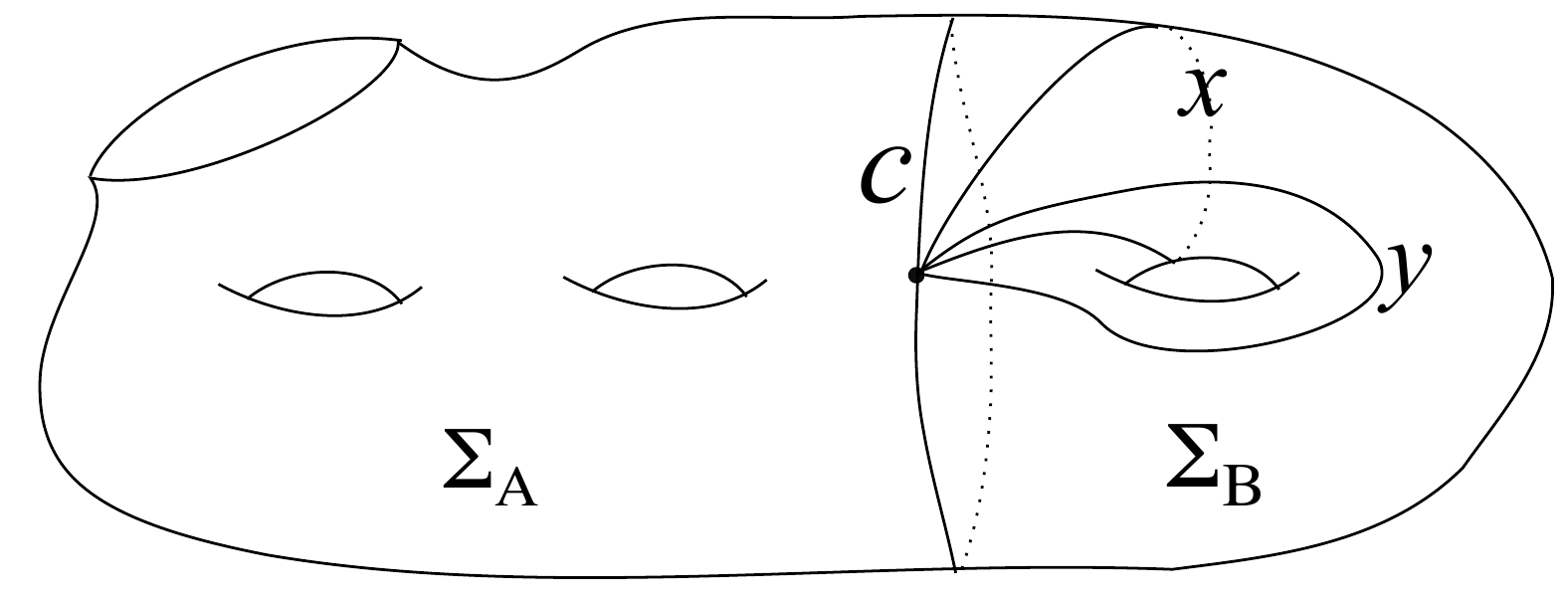}}}
 \caption{The setup: decomposition of $\Sigma$ and generators $x$ and $y$ for $B$}
  \label{setup}
  \end{figure*}

Let $x \in B$ and $y \in B$ be generators such that $B = \langle x, y \rangle$, and that $c$ represents the commutator $[x,y]$.   Fix $\alpha$ and $\beta$ in $\R \setminus \{0,\pm1\}$, and following \cite{DK} define $\phi_B: B \to \SL(2,\R)$ by 

$$\phi_B(x) = 
 \begin{pmatrix}
  \alpha & 0 \\
  0 & \alpha
 \end{pmatrix}$$
  $$\phi_B(y) = 
  \begin{pmatrix}
  \beta & 1 \\
  0 & \beta^{-1} 
 \end{pmatrix}$$  
  We have then 
  $$\phi_B([x, y]) =  \begin{pmatrix}
 1 & \beta(\alpha^2 - 1)  \\
  0 & 1  
 \end{pmatrix}$$
so that $\phi_B([x, y])$ is invariant under conjugation by the matrix $\lambda_t := \bigl( \begin{smallmatrix} 
 1 & t  \\
  0 & 1  
\end{smallmatrix} \bigr)$.  
  
Projecting this representation to $\PSL(2,\R)$ gives a representation which is upper triangular, hence solvable and therefore non-injective.  Abusing notation, let $\phi_B$ denote the representation to $\PSL(2,\R)$.
 
Now let $\phi_A : A \to \PSL(2, \R)$ be Fuchsian and such that the image of the boundary curve $c$ under $\phi_A$ agrees with $\phi_B([x,y])$.  Such a representation exists for the following reasons.  First, if $\Sigma$ has negative Euler characteristic, genus $g>1$, and is not the once punctured torus, then $\Sigma_A$ will have negative Euler characteristic as well and admit a hyperbolic structure.  Secondly, the Fuchsian representation coming from the hyperbolic structure will send the element $[c]$ representing the boundary curve to a parabolic, so after conjugation we may assume that it is equal to $\phi_B([x,y])$, provided $\phi_B([x,y])$ is parabolic, i.e. $\beta(\alpha^2 - 1) \neq 0$.  
 
Finally, combine $\phi_A$ and $\phi_B$ to get a one-parameter family of representations $\phi_t$ of $\pi_1(\Sigma) = A \ast_C B$ to $\PSL(2,\R)$ as follows.  For $t \in \R$ and $g \in A \ast_C B$, let

$$\phi_t(g) = \left\{ \begin{array}{rcl} \phi_A(g) & \mbox{if} & g \in A \\ \lambda_t \circ \phi_B(g) \circ (\lambda_t)^{-1} & \mbox{if} & g \in B
\end{array}\right.$$

This representation is well defined because $\phi_B([x, y]) = \phi_A([x,y])$ and $\phi_B([x, y])$ is invariant under conjugation by $\lambda_t$.  

Our next goal is to show that for appropriate choice of $\alpha, \beta$ and $t$, the representation $\phi_t$  satisfies the criteria in Theorem \ref{main}.  The main difficulty will be checking that no element representing a simple closed curve is of finite order.  To do so, we employ a stronger form of Lemma 2 from \cite{DK}.  This is:

\begin{lemma} \label{transcendental}
Suppose $w \in A \ast_C B$ is a word of the form $w = a_1b_1a_2b_2 ... a_lb_l$ with $a_i \in A$ and $b_i \in B$ for $1 \leq i \leq l$. Assume that for each $i$, the matrix $\phi_t(a_i)$ has a nonzero 2,1 entry and $\phi_t(b_i)$ is hyperbolic. If $t$ is transcendental over the entry field of $\phi_0(A \ast_C B)$, then $\phi_t(w)$ is not finite order.  
\end{lemma}

\noindent By \emph{entry field} of a group $\Gamma$ of matrices, we mean the field generated over $\Q$ by the collection of all entries of matrices in $\Gamma$.

\begin{remark} 
Lemma 2 of \cite{DK} is a proof that $\phi_t(w)$ is not the \emph{identity}, under the assumptions of Lemma \ref{transcendental}.    We use some of their work in our proof.  
\end{remark}

\begin{proof}[Proof of Lemma \ref{transcendental}]
In \cite{DK}, DeBlois and Kent show by a straightforward induction that the entries of $\phi_t(w)$ are polynomials in $t$, where the degree of the 2,2 entry is $l$, the degree of the 1,2 entry is at most $l$, and the other entries have degree at most $l-1$.  Now suppose that $\phi_t(w)$ is finite order.  Then it is conjugate to a matrix of the form 
$ \bigl( \begin{smallmatrix} 
 u & v  \\
  -v & u  
\end{smallmatrix} \bigr)$. 
where $u = \cos(\theta)$ and $v = \sin(\theta)$ for some rational angle $\theta$.  In particular, it follows from the deMoivre formula for sine and cosine that $u$ and $v$ are algebraic.  

Now suppose that the matrix conjugating $\phi_t(w)$ to 
$ \bigl( \begin{smallmatrix} 
 u & v  \\
  -v & u  
\end{smallmatrix} \bigr)$ has entries $a_{ij}$.  Then we have
$$ \phi_t(w) =  \begin{pmatrix}
u -  (a_{12}a_{22} -  a_{11}a_{21})v   &  (a_{12}^2  a_{11}^2) v \\
 -(a_{22}^2  a_{21}^2)v &   u + (a_{12}a_{22} +  a_{11}a_{21})v
 \end{pmatrix}$$
Looking at the 2,2 entry we see that $a_{12}a_{22} +  a_{11}a_{21}$ must be a polynomial in $t$ of degree $l$.  But this means that the 1,1 entry is also a polynomial in $t$ of degree $l$, contradicting Deblois and Kent's calculation.  This proves the lemma.
\end{proof}

To complete our construction, choose $t$ to be transcendental over the entry field of $\phi_0(A\ast_C B)$.  We want to show that no power of an element representing a simple closed curve lies in the kernel of $\phi_t$.  To this end, consider any word $w$ in $A \ast_C B$ that has a simple closed curve as a representative.  There are three cases to check.  First, if $w$ is a word in $A$ alone, then $\phi_t(w)$ is not finite order, since $\phi_t(A)$ is Fuchsian and therefore injective.  Secondly, if $w$ is a word in $B$, then an elementary geometric argument shows that $w$ can only be represented by a simple closed curve if it has one of the following forms: 
\begin{enumerate} 
\item $w=x ^{\pm 1}$ or $w= y^{\pm 1}$
\item $w = [x^{\pm1},y^{\pm1}]$
\item Up to replacing $x$ with $x^{-1}$, $y$ with $y^{-1}$ and interchanging $x$ and $y$, there is some $n \in \Z^+$ such that $w = x^{n_1}y x^{n_2} y ... x^{n_s}y$ where $n_i \in \{n, n+1\}$.  
\end{enumerate}
We leave this as an exercise for the reader.  
This classification of words representing simple closed curves in $\Sigma_B$ also follows from a much more general theorem in \cite{BS}.

By construction, no word of type 1, 2 or 3 is finite order provided that $\alpha^s \beta^k \neq 1$ for any integers $s$ and $k$ other than zero -- indeed, we only need to check words of type 3, and these necessarily have trace $\alpha^s\beta^k + \alpha^{-s}\beta^{-k}$ for some $s, k \neq 0$.  Note that in particular, under the condition that $\alpha^s \beta^k \neq 1$ for $s, k \neq 0$, all type 3 words are hyperbolic.  We will use this fact again later on.  

For the remaining case where $w$ is a word with letters in both $A$ and $B$, we claim that it can be written in a form where Lemma \ref{transcendental} applies.  To write it this way, use the following procedure: First take a simple representative $\gamma$ for $w$ and apply an isotopy so that each crossing of $\gamma$ with $c$ occurs in some small neighborhood of the basepoint $p$.  This gives us a well defined shortest path along $c$ to $p$ from each crossing.  After further isotopy, we may assume additionally that no segment of $\gamma$ together with the shortest path along $c$ from its endpoints to $p$ bounds a disc, and that $\gamma$ is transverse to $c$.  All this can be done without introducing any self-crossings in $\gamma$.  Now $\gamma$ is of the form $\gamma_1 \delta_1 \gamma_2 \delta_2 ... \gamma_l \delta_l$ where $\gamma_i$ is a simple arc in $\Sigma_A$ and $\delta_i$ a simple arc in $\Sigma_B$. Close each $\gamma_i$ and $\delta_i$ into a simple loop by connecting its endpoints to $p$ using the shortest path along $c$ and let $a_i \in A$ (respectively $b_i \in B$) be the corresponding element of the fundamental group.  See Figure \ref{fig1}.  This gives us a word $a_1b_1a_2b_2 ... a_lb_l$ equivalent to $w$ after cyclic reduction, and each $a_i$ is represented by a simple closed curve in $\Sigma_A$ and each $b_i$ by a simple closed curve in $\Sigma_B$.  The elimination of discs bounded between segments of $\gamma$ and short segments of $c$ ensures that each $a_i$ and $b_i$ is nontrivial.  

  \begin{figure*}
  \centerline{
    \mbox{\includegraphics[width=3.3in]{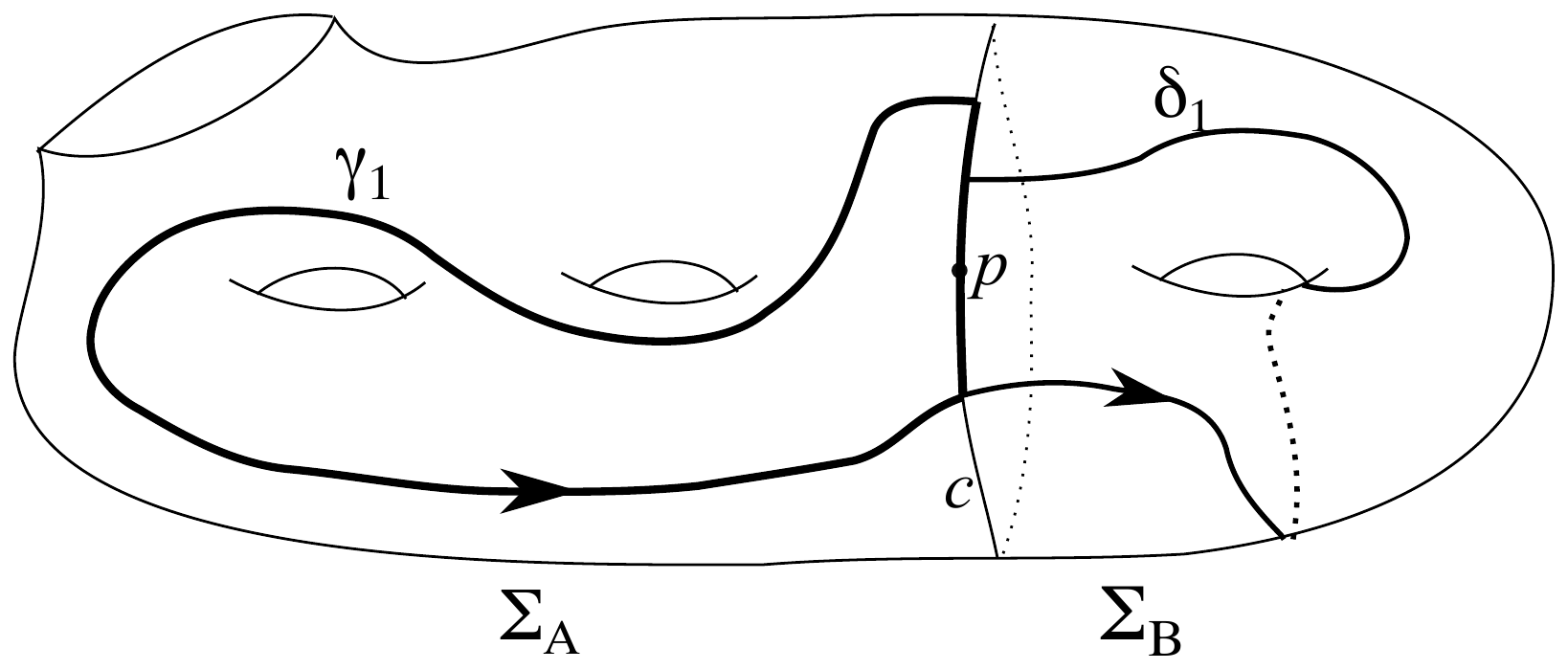}}}
 \caption{$a_1$ and $b_1$ in $w$, represented by $\gamma_i$ and $\delta_i$ joined to $p$}
  \label{fig1}
  \end{figure*}

We can also show that each $a_i$ either has a non-zero 2,1 entry or is represented by the curve $c$ or its inverse.  This is because $\phi_A$ is Fuchsian, so the only elements fixing infinity -- that is, with 2,1 entry equal to zero -- are powers of $c$, and no powers of $c$ other than $c^{\pm1}$ have a simple closed curve representative.  Similarly, the classification of words representing simple closed curves in $\Sigma_B$ shows that each $b_i$ is either hyperbolic or represented by $c$ or $c^{-1}$.  We claim that we may now rewrite $w$ to eliminate all appearances of $c$, keeping each $a_i$ with a non-zero 2,1 entry and each $b_i$ hyperbolic.  After doing so, we will have $w$ in a form where we can apply Lemma \ref{transcendental}.  

To rewrite $w$ in the desired form, first note that all $\gamma_i$ such that $a_i$ is represented by $c$ may be homotoped (simultaneously) into $\Sigma_b$ without introducing any self intersections of $\gamma$.  Thus, we can replace each such $\delta_{i-1} \gamma_i \delta_i$ with a simple loop $\delta'_i$ in $\Sigma_B$ alone, and rewrite $w = a_1b_1... a_{i-1} b'_{i} a_{i+1} ... a_lb_l$.  Reindex so that $w = a_1b_1a_2b_2 ... a_kb_k$ for $k < l$, and reindex the corresponding $\delta_i$ and $\gamma_i$ as well.  
Now repeat the procedure on this new word with each $b_i$: homotope all $\delta_i$ such that $b_i$ is represented by $c$ over to $\Sigma_A$ without introducing any self intersections of $\gamma$, and then replace each such $\gamma_i \delta_i \gamma_{i-1}$ with a simple loop $\gamma_i'$ in $\Sigma_B$ alone. Then rewrite $w$ so that, after reindexing, $w = a_1b_1a_2b_2 ... a_mb_m$ with $m<k$ and each $a_i$ and $b_i$ is a simple closed curve.  Repeat the process again with the $a_i$ of this new word.  The procedure ends when either no $a_i$ or $b_i$ is represented by $c$, or when $w$ is a word in $A$ or $B$ alone, represented by a simple loop in $\Sigma_A$ or $\Sigma_B$.  In the first case, Lemma \ref{transcendental} applies to show that $\phi_t(w)$ is not finite order.  In the second case, we have already shown that a word in $A$ or $B$ represented by a simple loop in $\Sigma_A$ or $\Sigma_B$ cannot be finite order.  

It remains only to remark that the representation $\phi_t$ is non-injective and that, by choosing appropriate parameters, we can produce uncountably many nonconjugate representations.   Non-injectivity follows immediately since $\phi_t(B)$ is solvable so the restriction of $\phi_t$ to $B$ is non-injective.  Now for any fixed $\alpha$ and $\beta$  (satisfying the requirement that $\alpha^s \beta^k \neq 1$ for all integers $s, k$), varying $t$ among transcendentals over the entry field of $\phi_0(A \ast_C B)$ produces uncountably many non-conjugate representations that are all non-injective, but have no power of a simple closed curve in the kernel.  This concludes the proof of Theorem \ref{main}, assuming that the surface was not the punctured torus.  

The punctured torus case is now immediate: any representation of the form of $\phi_B$ where $\alpha^s \beta^k \neq 1$ for any integers $s$ and $k$ is non-injective and our work above shows that no element represented by a simple closed curve has finite order.  Fixing $\alpha$ and varying $\beta$ produces uncountably many non-conjugate representations.






\end{document}